\documentclass[letter,10pt]{article}
\usepackage{amssymb}
\usepackage{amsthm}
\usepackage{amsmath}
\usepackage{mathrsfs}
\usepackage{verbatim}

\title{Optimal transportation with infinitely many marginals\footnote{The author was supported in part by a University of Alberta start-up grant.}}
\author{Brendan Pass\footnote{Department of Mathematical and Statistical Sciences, 632 CAB, University of Alberta, Edmonton, Alberta, Canada, T6G 2G1 pass@ualberta.ca.}}
\begin{document}
\maketitle
\begin{abstract}
We formulate and study an optimal transportation problem with infinitely many marginals; this is a natural extension of the multi-marginal problem studied by Gangbo and Swiech \cite{GS}.  We prove results on the existence, uniqueness and characterization of the optimizer, which are natural extensions of the results in \cite{GS}.  The proof relies on a relationship between this problem and the problem of finding barycenters in the Wasserstein space, a connection first observed for finitely many marginals by Agueh and Carlier \cite{ac}. 
\end{abstract}
\section{Introduction}
In this paper, we study an optimal transportation problem with infinitely many marginals.  

Optimal transportation with two marginals is an exciting and fast moving area of research.  The general goal is to couple two probability measures together as efficiently as possible, relative to a given cost function.  More precisely, given measures $\mu_1$ and $\mu_2$ (called \textit{marginals}) on topological spaces $M_1$ and $M_2$, respectively, and a cost function $c: M_1 \times M_2 \rightarrow \mathbb{R}$, the aim is to find the measure $\gamma$ on $M_1 \times M_2$ which projects to $\mu_1$ and $\mu_2$ and minimizes the total transportation cost:

\begin{equation*}
\int_{M_1 \times M_2} c(x_1,x_2) d\gamma
\end{equation*}

Equivalently, one can formulate this problem using more probabilistic language.  Here one looks for an $M_1 \times M_2$ valued random variable $(X_1,X_2)$, such that law$X_i =\mu_i$, for $i=1,2$, which minimizes the expectation:

\begin{equation*}
E [c(X_1,X_2)]
\end{equation*}
Results about the existence, uniqueness and structure of the optimal measure $\gamma$ have been proven for a wide class of cost functions and marginals; for a detailed review, see the monograph of Villani \cite{V2}.  A central theme is that, under certain conditions on the cost and the measures, there is a unique optimal measure $\gamma$, concentrated on the graph of a function, $x_2 = F(x_1)$; this was first proven for the quadratic cost $c(x_1,x_2)=|x_1-x_2|^2$ on $M_1=M_2=\mathbb{R}^n$ by Brenier \cite{bren} and was generalized to a large class of cost functions by Gangbo \cite{G}, Gangbo and McCann \cite{GM}, Caffarelli \cite{Caf}, McCann \cite{m} and Levin \cite{lev}.  In probabilistic terms, this means that the random variables $(X_1,X_2)$ are completely dependent. 

In recent years, optimal transportation problems with several marginals have started to attract more attention; this is a natural generalization of the preceding problem.  Give $m$ probability measures $\mu_1,\mu_2,...,\mu_m$ on topological spaces $M_1,M_2,...,M_m$, and a cost function $c: M_1\times M_2 \times ...\times M_m \rightarrow \mathbb{R}$, we look for the measure $\gamma$ on the product $M_1\times M_2 \times ...\times M_m$ which projects to the $\mu_i$ respectively, and minimizes

\begin{equation*}
\int_{M_1 \times M_2 \times ...\times M_m} c(x_1,x_2,...,x_m) d\gamma
\end{equation*}

As in the two marginal case, this problem may be formulated probabilistically.  In this setting, one looks for an $M_1\times M_2 \times ...\times M_m$ valued random variable $(X_1,X_2,...,X_m)$, such that law$(X_i)=\mu_i$ for $i=1,2,...,m$, minimizing 

\begin{equation*}
E [c(X_1,X_2,...,X_m)]
\end{equation*}

In contrast to the two marginal case, results concerning the structure of the optimal measure for $m>2$ are rather scarce.  However, Gangbo and Swiech proved that for the cost function $c(x_1,x_2,...x_m) = \sum_{i=1}^m\sum_{j=1}^m|x_i-x_j|^2$ on $M_i =\mathbb{R}^n$, the Kantorovich problem admits a unique solution which is concentrated on the graph of a function over the first marginal, generalizing Brenier's theorem \cite{GS}; see also \cite{OR}\cite{KS}\cite{RU} and \cite{RU2}.  As in the two marginal case, this means that the random variables $(X_1,X_2,...,X_m)$ are completely dependent.  Since then, a handful of results have been proven on the structure of solutions for different cost functions by Heinich \cite{H}, Carlier \cite{C}, Carlier and Nazaret \cite{CN} and the present author \cite{P}\cite{P1}.  Applications for multi-marginal optimal transportation have also arisen in mathematical economics \cite{ce}\cite{cmn} and condensed matter physics \cite{cfk}\cite{cfk2}.

Our goal in the present article is to study this problem in the limit as $m \rightarrow \infty$, restricting our attention to a cost function reminiscent of that of Gangbo and Swiech.  More precisely, we will prescribe a \textit{continuum} of probability measures $\mu_t$ on $\mathbb{R}^n$, for $t \in [0,1]$.  We will then look for the measurable\footnote{By definition, the stochastic process $X_t =X_t(\omega)$ is a mapping from $\Omega \times [0,1] \rightarrow \mathbb{R}^n$, where $\Omega$ is a probability space.  By measurable, we mean that this mapping is measurable, with respect to product measure on $\Omega \times [0,1]$; by Fubini's theorem, this implies that the sample paths $t\mapsto X_t$ are measurable almost surely.} stochastic process, $X_t$, with single time marginals law$(X_t)=\mu_t$, that minimizes

\begin{equation*}\tag{$MK_{\infty}$}
E(\int_0^1 \int_0^1 |X_s -X_t|^2dsdt) 
\end{equation*}

After expanding $|X_s -X_t|^2$ and noting that, by Fubini's theorem, $E(\int_0^1 X_t^2dt) = \int_0^1 E(X_t^2)dt = \int_0^1 \int_{\mathbb{R}^n}x^2d\mu_t(x)dt$, for \textit{any} measurable process such that law$X_t = \mu_t$ for all $t$, it is clear that this is equivalent to \textit{maximizing}:

\begin{equation*}
E((\int_0^1 X_tdt)^2)
\end{equation*}

We can think of the function $\int_0^1 \int_0^1 |X_s -X_t|^2dsdt$ as the limit of the Gangbo and Swiech cost.  On the other hand, for a sample path $X_t$, the integral $\int_0^1 X_t dt$ represents the average position of the sample path.  If we think of $X_t$ as representing a particle moving in a quadratic potential, over a time period $t \in [0,1]$, then $(\int_0^1 X_t dt)^2$ is the potential of the average position of the particle.

Our main result, Theorem \ref{character}, asserts existence and uniqueness of an optimizer in $(MK_{\infty})$, as well as a characterization of it, and is the natural generalization of the result of Gangbo and Swiech from finitely many to infinitely many marginals.  Roughly speaking, it says that the random curve $X_t$ is completely dependent, or deterministic; if $X_{t_0}$ is known for one fixed $t_0$, then $X_t$ is known for all $t$ (see Theorem \ref{monge}).

The typical approach to optimal transportation problems (with finitely many marginals) is to develop a duality theory, and then to use the resulting first order conditions to derive structural results about the optimal measure.  Our strategy here is quite different.  A recent paper of Agueh and Carlier relates the multi-marginal problem with Gangbo and Swiech's cost function to barycenters in the Wasserstein space \cite{ac}.  In this paper, we first generalize their results on existence, uniqueness and regularity from barycenters of finitely many points to barycenters of curves.  Having done this, we adapt their relationship between barycenters and multi-marginal problems to our setting and then exploit this connection to deduce the existence and uniqueness of the solution to our problem. 

Barycenters in the Wasserstein space are an interesting topic in their own right.  Barycenters of probability measures on general length spaces have attracted quite a bit of attention recently, in large part because of their relationship to curvature.  In spaces with Alexandrov curvature bounded above, the behaviour of barycenters is fairly well understood; see the work of Sturm \cite{sturm}.  The study of barycenters on spaces with curvature bounded below has recently been initiated by Ohta, and remains in its infancy \cite{ohta}.  It is, however, already apparent that barycenters on spaces with lower curvature bounds are not as well behaved as their counterparts on spaces with upper curvature bounds.  In particular, on spaces with non-positive curvature, each measure admits a unique barycenter, whereas on spaces with non-negative curvature, barycenters may be non-unique.  As an elementary example, every point on the equator is a barycenter of the north and south pole on the unit sphere.  

In addition, the problem of interpolating among several probability measures has begun to arise in applied problems including texture mixing \cite{bdpr} and mathematical economics \cite{ce}\cite{cmn}.  In fact, in \cite{ce}, the authors also consider an extension of their model which involves interpolating among an infinite number of measures. 

It is well known that the Wasserstein space over $\mathbb{R}^n$ does not have non-positive Alexandrov curvature  \cite{ags}.  The work of Agueh and Carlier provides uniqueness and regularity results, as well as a characterization of the barycenter of finitely many points in the Wasserstein space, under certain regularity conditions.  Our first contribution is to extend their uniqueness and regularity result to a continuous curve $\mu_t$ of measures.  It is also worth noting that our techniques here can be used to extend some of the results of Agueh and Carlier to other underlying spaces; in particular, we prove uniqueness of barycenters in the Wasserstein space over a Riemannian manifold.

Finally, let us mention that in a separate paper, we study infinite marginal optimal transportation for somewhat more general cost functions, restricted to the case $n=1$ \cite{P6}.  The techniques used there are quite different than here.

In the next section, we will introduce our hypotheses on the curve of measures, $\mu_t$, as well as two regularity assumptions which will be assumed only at specific points.  In the third section we will study the barycenter of the curve $\mu_t$ proving existence, uniqueness and regularity, as well as demonstrating that the uniqueness result can be extended to other settings.  In section 4, we develop the connection between barycenters and the problem $(MK_{\infty})$ and use this to prove existence and uniqueness of the optimal stochastic process. 
\section{Notation and assumptions}
  
We will denote by $P_2(\mathbb{R}^n)$ the set of all probability measures on $\mathbb{R}^n$ with finite second moments and $P_{ac,2}(\mathbb{R}^n)$ the subset of these which are absolutely continuous with respect to Lebesgue measure. For $\mu, \nu \in P_2(\mathbb{R}^n)$,  $W_2(\mu,\nu)$ denotes the quadratic Wasserstein distance between the measures $\mu$ and $\nu$:

\begin{equation*}
W_2^2(\mu,\nu) = \inf \int_{\mathbb{R}^n \times \mathbb{R}^n}|x-y|^2d\gamma(x,y)
\end{equation*}
where the infimum is taken over all Borel probability measures $\gamma$ on $\mathbb{R}^n \times \mathbb{R}^n$ projecting to $\mu$ and $\nu$, respectively.

Let $M \subseteq \mathbb{R}^n$ be a bounded domain and $P(M) \subseteq P_2(\mathbb{R}^n) $ be the set of all Borel probability measures on $M$.  We will assume that all of our measures $\mu_t$ are supported on $M$; that is, $\mu_t \in P(M)$.  We will denote by $c(M)$ the convex hull of $M$ and $P(c(M))$ the set of Borel probability measures on $c(M)$.  We will assume that $\mu_t$ is a weakly continuous curve in $P(M)$; that is, we assume the mapping $t\mapsto \mu_t$ is a continuous map with respect to the weak topology.  Note that, by the boundedness of $M$, this is equivalent to continuity with respect to the Wasserstein metric.

We now introduce two different regularity conditions on the $\mu_t$, which we will assume at different times.
\newtheorem*{assumption}{Assumption A}

\begin{assumption}
The set 
\begin{equation*}
A := \Big\{t: \mu_t =g_t(x)dx \text{ is absolutely continuous with respect to Lebesgue measure.}\Big\} 
\end{equation*}
has positive Lebesgue measure.

\end{assumption}

\newtheorem*{assumption2}{Assumption B}

\begin{assumption2}
The set 
\begin{equation*}
A_{\infty} := \Big\{t: \mu_t =g_t(x)dx \text{ is absolutely continuous with respect to Lebesgue measure and } ||g_t||_{L^{\infty}} < \infty \Big\} 
\end{equation*}
has positive Lebesgue measure.
\end{assumption2}
 
Note that assumption $B$ easily implies that for some $K < \infty$, the set 
\begin{equation*}
A_K := \Big\{t: \mu_t =g_t(x)dx \text{ is absolutely continuous with respect to Lebesgue measure and } ||g_t||_{L^{\infty}} \leq K\Big\} 
\end{equation*}
has positive Lebesgue measure.

Finally, note that Assumption B clearly implies Assumption A.

\section{Barycenters}
In this section, we study the barycenter of the measures $\mu_t$.  By definition, this is the minimizer of 

\begin{equation*}
\mu \mapsto \int_0^1W_2^2(\mu,\mu_t)dt \tag{$B\infty$}
\end{equation*}
over the set $P_2(\mathbb{R}^n)$.    In this section, we consider existence, uniqueness, and regularity of the barycenter.  We also prove generalizations for other distributions of measures and other underlying spaces. 
\subsection{Existence of the barycenter}
Rather than proving the existence of a barycenter directly, we will, loosely speaking, approximate $(B\infty)$ by 
\begin{equation*}
\mu \mapsto \sum_{i=1}^{N}W_2^2(\mu_{\frac{i}{N}},\mu)\tag{BN}
\end{equation*}
and take the limit as $N$ tends to infinity.  This approach will prove useful later when we establish the regularity of the barycenter.  

\newtheorem{exist}{Proposition}[subsection]
\begin{exist}
A barycenter (a measure $\mu \in P_2(\mathbb{R}^n)$ minimizing $(B\infty)$) exists and it is supported on the convex hull of $M$.
\end{exist}
\begin{proof}
The result of  Agueh and Carlier implies the existence of a minimizer $\mu^N$ for $(BN)$, as this is simply the barycenter for the measures $\mu_{\frac{i}{N}}$, $i=1,...N$, with equal weights.  They also prove that it is supported on the set $\sum_{i=1}^N\frac{1}{N}M$, which is contained in the convex hull $c(M)$.

This yields, for all $\nu \in P_2(\mathbb{R}^n)$,

\begin{equation*}
\sum_{i=1}^{N}W^2_2(\mu_{\frac{i}{N}},\mu^N) \leq \sum_{i=1}^{N}W^2_2(\mu_{\frac{i}{N}},\nu)
\end{equation*}
or,
\begin{equation*}
\frac{1}{N}\sum_{i=1}^{N}W^2_2(\mu_{\frac{i}{N}},\mu^N) \leq \frac{1}{N}\sum_{i=1}^{N}W^2_2(\mu_{\frac{i}{N}},\nu)
\end{equation*}
Consider now the sequence $\mu^N$; by Prokhorov's theorem and the tightness of the set $P(c(M))$, we can assume, up to extraction of a subsequence, that $\mu^{N}$ converges weakly.  This implies that $\mu^{N}$ converges in the Wasserstein metric, and so, letting $\mu^{\infty} \in P(c(M))$ be the weak limit, $W_2(\mu^N,\mu^{\infty}) \rightarrow 0$.

Now, by the triangle inequality, for any $\nu \in P_2(\mathbb{R}^n)$, we have

\begin{eqnarray}
\frac{1}{N}\sum_{i=1}^{N}W^2_2(\mu_{\frac{i}{N}},\mu^{\infty}) & \leq& \frac{1}{N}\sum_{i=1}^{N}[W_2(\mu_{\frac{i}{N}},\mu^N) +W_2(\mu^{\infty},\mu^N)]^2\nonumber\\
&=& \frac{1}{N}\sum_{i=1}^{N}W^2_2(\mu_{\frac{i}{N}},\mu^N) + \frac{1}{N}\sum_{i=1}^{N}W^2_2(\mu^{\infty},\mu^N)+2\frac{1}{N}\sum_{i=1}^{N}W_2(\mu^{\infty},\mu^N)W_2(\mu_{\frac{i}{N}},\mu^N)\nonumber\\
&\leq & \frac{1}{N}\sum_{i=1}^{N}W^2_2(\mu_{\frac{i}{N}},\nu) + \frac{1}{N}\sum_{i=1}^{N}W^2_2(\mu^{\infty},\mu^N) +2W_2(\mu^{\infty},\mu^N)\frac{1}{N}\sum_{i=1}^{N}W_2(\mu_{\frac{i}{N}},\mu^N)\nonumber\\
&=& \frac{1}{N}\sum_{i=1}^{N}W^2_2(\mu_{\frac{i}{N}},\nu) + W^2_2(\mu^{\infty},\mu^N) +2W_2(\mu^{\infty},\mu^N)\frac{1}{N}\sum_{i=1}^{N}W_2(\mu_{\frac{i}{N}},\mu^N) \label{last}
\end{eqnarray}
Now, as the compact set $c(M) \subseteq \mathbb{R}^n$ is bounded and  $\mu_{\frac{i}{N}}$ and $ \mu^N$ are supported on $c(M)$, we have, for some $M$, $|x-y|^2<M$, whenever $x \in $spt$(\mu^N)$ and $y \in $spt$(\mu_{\frac{i}{N}})$.  Therefore 
\begin{equation*}
W_2(\mu_{\frac{i}{N}},\mu^N) \leq \sqrt{M}
\end{equation*}
and so the last term in inequality (\ref{last}) is bounded above by 
\begin{equation*}
2W_2(\mu^{\infty},\mu^N)\sqrt{M}.
\end{equation*}
Now, as $N \rightarrow \infty$, $\mu^N \rightarrow \mu^{\infty}$ in the Wasserstein metric, and so the last two terms above tend to $0$.  As the curve $t \mapsto \mu_t$ is continuous with respect to the Wasserstein distance, the mapping $t \mapsto W^2_2(\mu_t,\mu^{\infty})$ is continuous, by the triangle inequality.  Therefore, the quantity on the left hand side of inequality (\ref{last}) tends to the Riemann integral of this curve as $N$ tends to $\infty$.  A similar conclusion holds for the first term on the right hand side, and so, taking the limit of  as $N \rightarrow \infty$ in inequality (\ref{last}) yields:
\begin{equation*}
\int_{0}^1W^2_2(\mu_t,\mu^{\infty})dt \leq \int_{0}^1W^2_2(\mu_t,\nu)dt
\end{equation*}
As this holds for any measure $\nu \in P_2(\mathbb{R}^n)$, this means that $\mu^{\infty}$ is a barycenter.
\end{proof}

\subsection{Uniqueness of the barycenter}
In this section, we establish uniqueness of the barycenter, under Assumption A.

\newtheorem{convexity}{Lemma}[subsection]
\begin{convexity}
Fix $\nu \in P_2(\mathbb{R}^n)$.   The function $P_2(\mathbb{R}^n) \ni \mu \mapsto W_2^2(\nu,\mu)dt$ is convex on $P_2(\mathbb{R}^n)$.  If $\nu$ is absolutely continuous with respect to Lebesgue, it is strictly convex.
\end{convexity}

Note that convexity here does \textit{not} mean displacement convexity in the sense of McCann \cite{m}; instead it means convexity with respect to the usual linear structure on the space of probability measures. This type of convexity is well known, and has been exploited in, for example, \cite{jko}.  To the best of my knowledge, however, the strict convexity has not been explored.

\begin{proof}
Choose two measures $\mu_0$ and $\mu_1$ in $P_2(\mathbb{R}^n)$.  For a fixed $t$, let $\gamma_{i}$ be optimal couplings between $\mu_i$ and $\nu$, for $i =0,1$, respectively.  Now, let $\mu_s = s\mu_1 +(1-s)\mu_0$ and set $\gamma_{s} = s\gamma_{1} +(1-s)\gamma_{0}$.  Note that $\gamma_{s}$ is a coupling of $\mu_s$ and $\nu$.  We then have

\begin{eqnarray}
W_2^2(\nu,\mu_s) &\leq& \int_{\mathbb{R}^n \times \mathbb{R}^n} |x -y|^2d\gamma_{s} \nonumber\\
&=&s\int_{\mathbb{R}^n \times \mathbb{R}^n} |x -y|^2d\gamma_{1} +(1-s)\int_{\mathbb{R}^n \times \mathbb{R}^n} |x -y|^2d\gamma_{0}\nonumber\\
&=& sW_2^2(\nu,\mu_1)+(1-s)W_2^2(\nu,\mu_0)
\end{eqnarray}

This establishes convexity of the function $\mu \mapsto W_2^2(\nu,\mu)$.  We now show that this convexity is strict if $\nu$ is absolutely continuous with respect to Lebesgue measure.  In this case, Brenier's theorem implies the existence of an optimal map $F_s: spt(\nu) \rightarrow spt(\mu_s)$ for each $s$, such that the unique optimal measure $\overline{\gamma_{s}}$ coupling $\nu$ and $\mu_s$ is concentrated on the graph $\{(x,F_s(x)\}$ \cite{bren}. 

Assume now that $\mu_0 \neq \mu_1$ and that $0< s < 1$; we need to show that the inequality above is strict.  Note first that the inequality is strict unless $\gamma_s$ is an \textit{optimal} coupling between $\nu$ and $\mu_s$.  

Now, as $\mu_0 \neq \mu_1$, the set $\{x: F_0(x) \neq F_1(x)\}$ has positive measure.  Note that for each $x$ where $F_0(x) \neq F_1(x)$, the coupling $\gamma_{s}$ \textit{splits} the mass at the point $x$ between $F_0(x)$ and $F_1(x)$; for such $x$, both $(x,F_0(x))$ and $(x,F_1(x))$ belong to the support of the measure $\gamma_{s}$.  On the other hand, the \textit{optimal} measure $\overline{\gamma_{s}}$ coupling $\nu$ and $\mu_s$ is concentrated on the graph of a function $F_s$, and so, for $\nu$ almost all $x$, there is only one point $(x,y)$ in the support of the optimizer (namely $(x,F_s(x))$).  This immediately implies that $\gamma_{s}$ is not the optimal coupling of $\nu$ and $\mu_s$ and so we must have a strict inequality.  This completes the proof.

\end{proof}
The preceding lemma easily implies the following result.
\newtheorem{intconvexity}[convexity]{Lemma}
\begin{intconvexity}
The function $\mu \mapsto \int_0^1W_2^2(\mu_t,\mu)dt$ is convex on $P_2(\mathbb{R}^n)$.  If Assumption A is satisfied, the function is strictly convex.
\end{intconvexity}
\begin{proof}
Let $\mu_0, \mu_1 \in P_2(\mathbb{R}^n)$.  The preceding lemma implies that for all $t \in [0,1]$ and all $s \in (0,1)$ we have 

\begin{equation*}
W_2^2(\mu_t,\mu_s) \leq sW_2^2(\mu_t,\mu_1)+(1-s)W_2^2(\mu_t,\mu_0)
\end{equation*}
and the inequality is strict on a subset of $[0,1]$ of positive measure.  Integrating with respect to $t$ yields the desired result.
\end{proof}

This result immediately implies the uniqueness of the barycenter.
\newtheorem{unique}[convexity]{Corollary}
\begin{unique}
Under regularity Assumption A, the barycenter is unique.
\end{unique}

\subsection{Regularity of the barycenter}
In this subsection we obtain a regularity result on the barycenter $\mu^{\infty}$, which will be crucial to our construction on the optimal stochastic process in section 4.  Agueh and Carlier \cite{ac} proved the following regularity result for the barycenter of finitely many measures (ie, a minimizer of (BN)); assume that, for at least one $i \in \{1,2,3,...,N\}$, the measure $\mu_{\frac{i}{N}}$ is absolutely continuous with respect to Lebesgue measure, with an $L^{\infty}$ density $g_i$.  Then the barycenter $\mu^N$ is absolutely continuous with an $L^{\infty}$ density $g^N$ and 

\begin{equation*}
||g^N||_{L^{\infty}} \leq N ||g_i||_{L^{\infty}}
\end{equation*}
Our general strategy in this subsection is to approximate  $\mu^{\infty}$ by barycenters $\overline{\mu^N}$ of finitely many measures, much like in subsection 3.1, and then deduce a regularity result (ie, a bound on the $L^{\infty}$ norm of the density) from the regularity of the $\overline{\mu^N}$.  Of course, the bound above tends to infinity as $N$ tends to infinity and so to accomplish this goal we will need a refined regularity result on the barycenters of finitely many measures, with a bound on the $L^{\infty}$ norm which is uniform in $N$.

\newtheorem{normbound}{Proposition}[subsection]
\begin{normbound}\label{normbound}
Let $\mu_i \in P_2(\mathbb{R}^n)$ for $i=1,2,....,N$ and suppose $\mu$ minimizes $\mu \mapsto \sum_{i=1}^N\lambda_iW^2_{2}(\mu,\mu_i)$ on $P_2(\mathbb{R}^n)$, where $0 <\lambda_i<1$ and $\sum_{i=1}^{\infty}\lambda_i=1$.  Let $B \subseteq \{1,2,..,.N\}$ be nonempty and assume that for all $i \in B$, $\mu_i$ is absolutely continuous with respect to Lebesgue measure with an $L^{\infty}$ density $g_i$.   Then $\mu$ is absolutely continuous with respect to Lebesgue measure with an $L^{\infty}$ density $g$ satisfying:

 \begin{equation*}
 ||g||_{L^{\infty}}\leq \Bigg[\sum_{i \in B}\frac{\lambda_i}{||g_i||_{L^{\infty}}^{\frac{1}{n}}}\Bigg]^{-n}
 \end{equation*}
\end{normbound}

\begin{proof}
By a result of Agueh and Carlier, (\cite{ac}, Proposition 3.8), for almost all $x$,  $\sum_{i=1}^N \lambda_i D u_i(x) =x$, where $D u_i$ is the Brenier map pushing the Barycenter $\mu^N$ forward to $\mu_i$. As each convex function $u_i$ is twice differentiable almost everywhere, we can differentiate this equation to obtain 
\begin{equation*}
\sum_{i=1}^N \lambda_iD^2u_i(x) =I_n
\end{equation*}
for almost all $x$, where $I_n$ is the $n \times n$ identity matrix.  Taking determinants and $n$-th roots yields:
\begin{equation*}
[\det\sum_{i=1}^N \lambda_iD^2u_i(x)]^{\frac{1}{n}} =1
\end{equation*}
As each $D^2u_i(x)$ is symmetric and positive definite wherever it exists, Minkowski's determinant inequality combined with the preceding equation yields:

\begin{equation*}
\sum_{i=1}^N \lambda_i(\det D^2u_i(x))^{\frac{1}{n}} \leq [\det\sum_{i=1}^N \lambda_iD^2u_i(x)]^{\frac{1}{n}}= 1.
\end{equation*}
As each term $\lambda_i(\det D^2u_i(x))^{\frac{1}{n}}$ is non-negative, we obtain 
\begin{equation*}
\sum_{i \in B}\lambda_i (\det D^2u_i(x))^{\frac{1}{n}} \leq \sum_{i=1}^N \lambda_i(\det D^2u_i(x))^{\frac{1}{n}} \leq 1.
\end{equation*}
 
Now, using the result of Agueh and Carlier, we know that $\mu$ is absolutely continuous, ie $d\mu = g(x)dx$, and it is well known that for each $i \in B$, $u_i$ solves the Monge-Ampere equation almost everywhere, $\det D^2u_i(x) = \frac{g(x)}{g_i(Du_i(x))}$.  Combined with the preceding inequality, this implies

\begin{equation*}
[g(x)]^{\frac{1}{n}} \leq \Bigg[\sum_{i\in B}\frac{\lambda_i}{g_i(Du_i(x))^{\frac{1}{n} }}\Bigg]^{-1}\leq \Bigg[\sum_{i \in B}\frac{\lambda_i}{||g_i||_{L^{\infty}}^{\frac{1}{n}}}\Bigg]^{-1}
\end{equation*}

\end{proof}

\newtheorem{bound}[normbound]{Lemma}
\begin{bound}
Assume Assumption B and let $m_K>0$ be the Lebesgue measure of the set $A_K$.  Then there exists a sequence of measures $\overline{\mu}^N$, absolutely continuous with respect to Lebesgue measure, with densities $\overline{g}^N(x)$ satisfying  $||\overline{g}^N||_{L^{\infty}}\leq \frac{K}{m_K^{n}}$, converging weakly to $\mu^{\infty}$.
\end{bound}
\begin{proof}
For $i=0,1,....,N-1$, set $I_i = [\frac{i}{N},\frac{i+1}{N}]$.  Let $B_K=\{i:I_i\cap A_K \neq \phi\}$ be subset of indices $i$ for which $I_i$ contains at least one point in the set $A_K$. The union $\bigcup_{i\in B_K }I_i$ clearly covers $A_K$, and so, denoting the size of $B_K$ by $|B_K|$, we must have 
\begin{equation*}
\frac{|B_K|}{N} \geq m_K.
\end{equation*}
Now, we choose $t_i \in I_i$ and approximate the Riemann integral much like in the proof of existence, except that, whenever $i$ is in $B_K$, we choose the point $t_i \in  I_i \cap A_K$, rather than taking $t_i =\frac{i}{N}$.  We define $\overline{\mu}^N$ to be the barycenter of the measures $\mu_{t_i}$, with equal weights $\lambda_i=\frac{1}{N}$; that is the minimizer on $P_2(\mathbb{R}^n)$ of:

\begin{equation*}
\mu \mapsto \frac{1}{N}\sum_{i=1}^{N}W^2_2(\mu_{t_i},\mu)
\end{equation*}

Our result above implies that the barycenter $\overline{\mu^N}$ is absolutely continuous with respect to Lebesgue measure, with a density $\overline{g}^N(x)$ satisfying
\begin{eqnarray*}
 ||\overline{g}^N||_{L^{\infty}} &\leq&\Bigg[\sum_{i \in B_K}\frac{\lambda_i}{||g_{t_i}||_{L^{\infty}}^{\frac{1}{n}}}\Bigg]^{-n}\\
 &\leq& \Bigg[\sum_{i \in B_K}\frac{1}{NK^{\frac{1}{n}}}\Bigg]^{-n}\\
 &=&  \Bigg[\frac{|B_K|}{NK^{\frac{1}{n}}}\Bigg]^{-n}\\
  &=&  \frac{N^nK}{|B_K|^{n}}\\
 &\leq& \frac{K}{m_K^n}
\end{eqnarray*}
Now, up to extraction of a subsequence, $\overline{\mu}^N$, converges weakly by Prokhorov's theorem to some measure $\overline{\mu}^{\infty}$.  Exactly as in the proof of existence, one can prove that $\overline{\mu}^{\infty}$ is a barycenter.  It then follows by the uniqueness result in the last subsection that $\overline{\mu}^{\infty} =\mu^{\infty}$.
\end{proof}
By approximation, we then easily obtain the following regularity result on our barycenter $\mu^{\infty}$.
\newtheorem{infinitlinfty}[normbound]{Corollary}
\begin{infinitlinfty}\label{infinitlinfty}
Under Assumption B, the barycenter is absolutely continuous with respect to Lebesgue measure and its density $g^{\infty}(x)$ satisfies  $||g^{\infty}||_{L^{\infty}}\leq \frac{K}{m_K^{n}}$.
\end{infinitlinfty}
\begin{proof}
It suffices to prove $\mu^{\infty}(A) \leq \frac{K}{m_K}|A|$ for any Borel set $A \subseteq c(M)$.  If $A$ is open, this follows easily from Lemma \ref{normbound}, as the weak convergence $\overline{\mu}^N \rightarrow \mu^{\infty}$ implies 

\begin{equation*}
\mu^{\infty}(A) \leq \liminf \overline{\mu}^N(A) \leq \frac{K}{m_K}|A|.
\end{equation*}
If $A$ is not open, we may, for any $\epsilon >0$, find an open set $U$ such that $A\subseteq U$ and $|U \setminus A| \leq \epsilon$.  Then we have 

\begin{eqnarray*}
\mu^{\infty}(A) & \leq & \mu^{\infty} (U) \\
&\leq &\frac{K}{m_K}|U| \\
&=&\frac{K}{m_K}( |A|+|U \setminus A|) \\
& \leq &\frac{K}{m_K}|A| +\frac{K}{m_K}\epsilon
\end{eqnarray*}
Taking the limit as $\epsilon \rightarrow 0$ yields the desired result.

\end{proof}

\subsection{Generalization to other spaces and distributions}
The purpose of this subsection is to demonstrate that our approach to uniqueness of the barycenter holds for more general underlying spaces $M$ and more general distributions of measures.  The results of this subsection are not essential to the rest of the paper and can safely be skipped.

For this subsection only, let $(M,g)$ be a compact Riemannian manifold and $P(M)$ denote the set of Borel probability measures on $M$.  Given probability measures $\mu$ and $\nu$ in $P(M)$, the Wasserstein distance between $\mu$ and $\nu$ is defined as in Euclidean space, with the Riemannian distance squared replacing the Euclidean distance:

\begin{equation*}
W_2^2(\mu,\nu) = \inf\int_{M \times M}d^2(x,y)d\gamma(x,y),
\end{equation*}
where the infimum is over all measures $\gamma$ on $M \times M$ projecting to $\mu$ and $\nu$, respectively.
 
Now, let $\Gamma$ be a probability measure on $P(M)$.  A barycenter of $\Gamma$ is a minimizer of 

\begin{equation*}
\mu \mapsto \int W_2^2(\mu,\nu)d\Gamma(\nu).
\end{equation*}

By continuity and Prokhorov's theorem, it is straightforward to verify that a barycenter exists; see, for example, \cite{ohta}.  We note here that our proof of uniqueness relied only on existence and uniqueness of Monge solutions for arbitrary $\mu$ and a set of $\nu$ of positive $\Gamma$ measure.  Let $P_{ac}(M)$ be the set of Borel probability measures on $M$ which are absolutely continuous with respect to local coordinates.  By McCann's theorem \cite{m3}, whenever $\nu \in P_{ac}(M)$, there is a unique optimal map between $\nu$ and $\mu$.   Therefore, we obtain:

\newtheorem{gen}{Theorem}[subsection]
\begin{gen}
Suppose that $P_{ac}(M) \subseteq P(M)$ has positive $\Gamma$ measure.  Then the barycenter of $\gamma$ is unique.
\end{gen}
Assuming $M$ is a bounded subset of $\mathbb{R}^n$, when $\Gamma$ has finite support, this yields the uniqueness theorem of Agueh and Carlier.  When $\Gamma$ is supported on a Wasserstein continuous curve, we recover our results from a previous section. 
\section{Infinitely many marginals}

\subsection{Construction and basic properties of the optimal process}
We now return to our problem of primary present interest, namely the optimal transportation problem with infinitely many marginals, $(MK_{\infty})$. 

We will use the barycenter from the previous section to construct a stochastic process $X_t^{opt}$.  We will then show that this process is the unique minimizer in $(MK_{\infty})$.

We construct our optimal process $X_t^{opt}$ as follows. 
\newtheorem{process}{Definition}[subsection]
\begin{process}\label{construct}
 We take our underlying probability space to be $\mathbb{R}^n$, with the barycenter $\mu^{\infty}$.  Then taking $D u_t$ to be the Brenier map pushing $\mu^{\infty}$ forward to $\mu_t$, we define a stochastic process by 
 \begin{equation*}
 X_t^{opt} (x)= D u_t(x),
 \end{equation*}
 with the barycenter $\mu^{\infty}$ as the underlying probability space. 
\end{process}
Note that this definition means that the sample paths of the optimal process $ X_t^{opt}$ are $t \mapsto D u_t(x)$, for $x \in c(M)$, with a probability given by $\mu^{\infty}$.
Recall that a stochastic process $Y_t$ is continuous in probability (or in measure) if, for all $t \in [0,1]$ and all $\epsilon >0$.
  
\begin{equation*}
\lim_{s \rightarrow t}\mathbb{P}(|Y_s -Y_t|>\epsilon) =0
\end{equation*}

To prove measurability of $X^{opt}_t$, we will need the following proposition.
\newtheorem{reg2}[process]{Proposition}
\begin{reg2}
The process $X_t^{opt}$ is continuous in probability.
\end{reg2}
\begin{proof}
Recall that the path $\mu_t$ is weakly continuous; this then easily follows from a well known result on the stability of optimal transportation (see, for example, Villani \cite{V2}, Corollary 5.23).
\end{proof}
Recall that a stochastic process $Y_t$ is a \textit{version} of another process $Z_t$ if, for all $t \in [0,1]$, $Y_t = Z_t$ almost surely.  In this case, we say that $Y_t$ and $Z_t$ are \textit{stochastically equivalent}.  It is a well known result that every stochastic process which is continuous in probability has a measurable version, and so the preceding proposition implies:
\newtheorem{meas}[process]{Corollary}

\begin{meas}
The process $X_t^{opt}$ has a measurable version.
\end{meas}
In light of the preceding corollary, we assume from now on that the process $X_t^{opt}$ is measurable.
\subsection{Proof of optimality}
Our aim is now to prove that the process $X_t^{opt}$ defined in the last subsection is in fact optimal for $(MK_{\infty})$.  As a preliminary step in this direction, we will need to show that the \textit{average measure}, defined by 
\begin{equation*}
\mu^{opt}_a = \text{law}(\int_0^1X_t^{opt} dt) 
\end{equation*}
is in fact the barycenter $\mu^{\infty}$.
\newtheorem{average}{Proposition}[subsection]

\begin{average}
Assume Assumption B.  Then, for the process $X_t^{opt}$ from Definition \ref{construct}, the average measure coincides with the barycenter: $\mu^{opt}_a = \mu^{\infty}$.
\end{average}

\begin{proof}
By Corollary \ref{infinitlinfty}, the barycenter $\mu^{\infty}$ is absolutely continuous with respect to Lebesgue measure.  Now, for each $t$, the function 
\begin{equation*}
\mu \rightarrow W_2^2(\mu_t, \mu),
\end{equation*}
restricted to the set $P_{ac,2}(\mathbb{R}^n)$ of absolutely continuous measures with finite second moments, is differentiable with respect to the Wasserstein structure on $P_{ac,2}(\mathbb{R}^n)$ \cite{ags}.  This means that given a curve $\mu_s$ in $P_{ac,2}(\mathbb{R}^n)$ with $\mu_0 =\mu^{\infty}$, we have:

\begin{equation*}
\frac{d}{ds}\big|_{s=0} W_2^2(\mu_s,\mu_t) =2\int_{\mathbb{R}^n} <y-D u_t(y), \xi_0 (y)>d\mu^{\infty}(y)
\end{equation*}
where $\xi_s(y)$ is a vector field satisfying $\frac{\partial \mu_s}{ds} + D\cdot(\mu_s \xi_s(y)) =0$; that is, the tangent to $\mu_s$ in $P_{ac,2}(\mathbb{R}^n)$.  Note that we are abusing notation slightly by identifying the measure $\mu_s$ with its density.

Using the dominated convergence theorem, this means that 
\begin{equation*}
\mu \mapsto \int_{0}^1W_2^2(\mu_t, \mu)dt
\end{equation*}
is differentiable on $P_{ac,2}(\mathbb{R}^n)$ as well, and so it's derivative must vanish at the minimizer, $\mu^{\infty}$.  Using the formula for the derivative, we have,

\begin{equation*}
0 = \int_{0}^1\int_{\mathbb{R}^n}<y-D u_t(y), \xi_0(y)>d\mu^{\infty}(y)dt
\end{equation*}
for any tangent vector field $\xi_0$

By Fubini's theorem,

\begin{equation}\label{der}
0 = \int_{\mathbb{R}^n}\Big<\Big(\int_{0}^1y-D u_t(y)dt\Big), \xi_0(y)\Big> d\mu^{\infty}(y)
\end{equation}
Note this holds for all tangent vector fields $\xi_0$ to $P_{ac,2}(\mathbb{R}^n)$ at $\mu^{\infty}$.  As each $u_t$ is a convex function, the integral $v(x) = \int_0^1u_t(x)dt$ is also convex, and again using the dominated convergence theorem we have:
\begin{equation*}
Dv(x) =\int_0^1Du_t(x)dt
\end{equation*}
In particular, we can take $\xi_0(y)=y -\int_0^1Du_t(y)dt =\int_0^1(y -Du_t(y))dt$ in (\ref{der}), to obtain

\begin{equation*}
0 = \int_{\mathbb{R}^n}\Big<\Big(\int_{0}^1y-D u_t(y)dt\Big), \Big(\int_0^1y -Du_t(y)dt\Big)\Big> d\mu^{\infty}(y) = \int_{\mathbb{R}^n}\Big|\Big(\int_{0}^1y-D u_t(y)dt\Big)\Big|^2 d\mu^{\infty}(y)
\end{equation*}
This implies
\begin{equation*}
0 = \int_{0}^1y-D u_t(y)dt
\end{equation*}
$\mu^{\infty}$ almost everywhere.  Therefore, $y \mapsto \int_0^1D u_t(y)dt = \int_0^1X^{opt}_t(y)dt$ is the identity mapping.  As this map pushes $\mu^{\infty}$ forward to $\mu_a$, this immediately implies the desired result. 

\end{proof}

\newtheorem{character}[average]{Theorem}
\begin{character}\label{character}
Assume Assumption B.  Then the $X_t^{opt}$ from Definition \ref{construct} is optimal for $(MK_{\infty})$.  It is the unique optimizer in the sense that, if $Y_t$ is any other optimal process, we have for almost all $t$, $X_t^{opt} = Y_t$, almost surely.
\end{character}
\begin{proof}
It is clear from the construction that law$X_t^{opt} = \mu_t$ for all t.
Now, take any stochastic process $Y_t$, such that law$(Y_t)=\mu_t$.  Denote by $\mu_{a}$ the law of the random variable $Y_a=\int_0^1Y_tdt$.  We will denote by $\mu_{t,a}$ the law of the ordered pair $(Y_t,Y_a)$ on $\mathbb{R}^n \times \mathbb{R}^n$; note that this implies that $\mu_t$ and $\mu_a$ are the marginals of $\mu_{t,a}$. 
Now, note that:

\begin{eqnarray*}
E\Big(\int_0^1\big(|Y_t -\int_0^1Y_sds|^2\big)dt\Big) &=& E\Big(\int_0^1|Y_t|^2dt -2\int_0^1Y_tdt\int_0^1Y_sds + \big|\int_0^1Y_sds\big|^2\Big)\\
&=& E\Big(\int_0^1|Y_t|^2dt - \big|\int_0^1Y_sds\big|^2\Big)\\
&=&\int_0^1E|Y_t|^2dt -E\big|\int_0^1Y_sds\big|^2\\
&=&\int_0^1\int_{\mathbb{R}^n}|Y_t|^2d\mu_tdt-E\big|\int_0^1Y_sds\big|^2\\
\end{eqnarray*}
Note that the first term above depends only on law$(Y_t) = \mu_t$.  Therefore, maximizing $E(|\int_0^1Y_sds|^2)$ subject to the constraint law$(Y_t) = \mu_t$ is equivalent to minimizing $E\Big(\int_0^1\big(|Y_t -\int_0^1Y_sds|^2\big)dt\Big)$, subject to the same constraint.  

Now, we have, using Fubini,
\begin{eqnarray*}
E\Big(\int_0^1(|Y_t -\int_0^1Y_sds|)dt\Big) &=&\int_0^1\Big(E\big(|Y_t -\int_0^1Y_sds|^2\big)\Big)dt\\
&=&\int_0^1\Big(\int_{\mathbb{R}^n \times \mathbb{R}^n}|Y_t -Y_a|^2d\mu_{t,a}\Big)dt\\
& \geq & \int_0^1 W^2_2(\mu_t, \mu_a)dt\\ 
& \geq & \int_0^1 W^2_2(\mu_t, \mu^{\infty})dt\\ 
\end{eqnarray*}

Observe that we have equality if and only if 
\begin{enumerate}
\item $\mu_a$ is the barycenter $\mu^{\infty}$ of the $\mu_t$'s and, 
\item for almost all $t$, the measure $\mu_{t,a}$ is the optimal coupling between $\mu_t$ and $\mu_a$.
\end{enumerate}

Now, assuming the first condition,  $\mu_a = \mu^{\infty}$ is absolutely continuous with respect to Lebesgue measure by Corollary \ref{infinitlinfty}, and so the optimal coupling $\mu_{t,a}$ is concentrated on the graph of the function $x \mapsto D u_t(x)$. Therefore, these two conditions imply that the sample path $Y_t$ is completely determined almost surely by $Y_a$, which is distributed according to the barycenter.  We can therefore take the underlying probability space to be $\mu^{\infty}$ and the second condition implies that the process $Y_t =X_t^{opt}$ almost surely, for almost all $t$.
\end{proof}

We can obtain a more elegant uniqueness result if we restrict our attention to stochastic processes which are continuous in probability;  the following theorem implies that $X_t^{opt}$ is the unique, continuous in probability maximizer for ($MK_{\infty}$), modulo stochastic equivalence.

\newtheorem{uniquereg}[average]{Theorem}
\begin{uniquereg}
Assume Assumption B and suppose $Y_t$ is optimal for ($MK_{\infty}$) and $Y_t$ is continuous in probability.  Then $Y_t$ is a version of $X_t^{opt}$.
\end{uniquereg}
\begin{proof}
From our previous uniqueness result, we know that $Y_t = X_t^{opt}$ almost surely, for almost all $t$.  We need to prove this for \textit{all} t.  

Fix $t_0 \in [0,1]$.  Then we can choose a sequence $t_i$ converging to $t_0$ such that $Y_{t_i} = X_{t_i}$ almost surely.  By continuity in probability, $Y_{t_i}$ converges to $Y_{t_0}$ in probability and $X_{t_i}$ converges to 
$X_{t_0}$ in probability.  This immediately implies $X_{t_0} = Y_{t_0}$ almost surely, as desired.
\end{proof}

We now prove an analogue of Brenier's theorem \cite{bren} (for two marginal problems) and the result of Gangbo and Swiech \cite{GS} (for several marginals).  In our context, it is natural to interpret this result as saying that we can take the underlying probability space of our stochastic process to be $M \subseteq \mathbb{R}^n$.

\newtheorem{monge}[average]{Theorem}
\begin{monge}\label{monge}
(Monge solutions)  Assume Assumption B holds and suppose $\mu_{t_0} \in P_{ac}(M)$ .  Then we can take $\mu_{t_0}$ to be the underlying probability space of the unique optimal process $X_t^{opt}$.  That is, the optimal process can be written as $X_t^{opt} =F_t(X^{opt}_{t_0})$, where, for each $t$,  $F_t:\mathbb{R}^n \rightarrow \mathbb{R}^n$ is a mapping pushing $\mu_{t_0}$ forward to $\mu_{t}$, and $F_{t_0}$ is the identity mappings.
\end{monge}

\begin{proof}
As  $\mu_{t_0}$ does not charge small sets, Brenier's theorem implies that the optimal map $Du_{t_0}$ pushing the barycenter $\mu^{\infty}$ forward to $\mu_{t_0}$ is invertible almost everywhere; its inverse is $Du^{*}_{t_0}$, where $u^{*}_{t_0}$ is the Legendre transform of $u$.   We then have $X_t^{opt}(x) = D u_t(x)=D u_t(D u^*_{t_0}(z))$, where $x$ is distributed according to $\mu^{\infty}$ and $z =D u_{t_0}(x)$ is distributed according to $\mu_{t_0}$. 
Taking $F_t =D u_t\circ D u^{*}_{t_0}$ yields the desired result.
\end{proof}
This result means that the stochastic process $X_t^{opt}$ is \textit{deterministic}, in the sense that if we know $X^{opt}_{t_0}$, we know $X_t^{opt}$ for all $t$.

\bibliographystyle{plain}
\bibliography{biblio}

\end{document}